\documentclass{birkjour}
 \newtheorem{thm}{Theorem}[section]
 \newtheorem{cor}[thm]{Corollary}
 \newtheorem{lem}[thm]{Lemma}
 
 \theoremstyle{definition}
 \newtheorem{defn}[thm]{Definition}
 \theoremstyle{remark}
 
 \newtheorem*{ex}{Example}
 \numberwithin{equation}{section}

\begin{document}

%
%
%
%
%
%
%
%
%

\title[Bounded composition operator on Lorentz spaces]
 {Bounded composition operator on Lorentz spaces}

\author{Nikita Evseev}

\address{%
Sobolev Institute of Mathematics\\
4 Acad. Koptyug avenue\\
630090 Novosibirsk\\
Russia}

\email{evseev@math.nsc.ru}

\thanks{This research was carried out at the Peoples' Friendship University of Russia
and financially supported by the Russian Science Foundation (Grant 16-41-02004)}


\date{January 31, 2017}

\begin{abstract}
We study a composition operator on Lorentz spaces.
In particular we provide necessary and sufficient conditions under which a measurable
mapping induces a bounded composition operator.
\end{abstract}

\maketitle
\section{Introduction}
Lorentz spaces $L_{p,q}$ are a generalization of ordinary Lebesgue spaces $L_p$,
and they coincide with $L_p$ when $q=p$. 
Some references as to basics on Lorentz spaces may be found in \cite{SW, M, G}.

A composition operator induced by map $\varphi$ on some function space is quite a natural object
which is defined as $C_\varphi f = f\circ\varphi$. 
Depending on the structure of particular function space various properties of a compositions operator 
are under interest e.g. boundedness, compactness, inevitability and so on.  
The study of composition operators may be divided into three directions. 

The first one could be referred to as classical and goes back to 
Littlewood’s Subordination Principle (1925).
This principle states that a holomorphic  self-mapping of the unit disk $D\subset C$
preserving $0$ induces a contractive composition operator on Hardy space $H^p(D)$,
as well on Bergman and Dirichlet spaces.	 
However, it is believed that the systematic study of composition operators
induced by holomorphic maps started with the paper \cite{N} by E. A. Nordgren in the mid 1960's. 
Afterwards the study of composition operator developed 
at the juncture of analytic function theory and operator theory.
We refer the reader to book \cite{S} by J. Shapiro.

The second direction has a more operator flavor.    
Researchers raised all the questions about composition operators which could be 
posed regarding operators on normed spaces. 
%
One may find an exhaustive survey on the topic in the book \cite{SM} by R. K. Singh, J. S. Manhas 
and also in the proceedings~\cite{SonCO}.

The survey on composition operators on Sobolev spaces was motivated by
the question, \textit{what change of variables does preserve a Sobolev class?} 
Therefore the research was primarily focused on analytic and geometric properties
of mappings, whereas operator theory was involved to a lesser extent.   	
The first results in this area are due to S. L. Sobolev (1941), 
V. G. Maz'ya (1961), F. W. Gehring (1971).
Subsequently S.~K.~Vodopjanov and V.~M.~Gold{\v{s}}te{\u\i}n (1975-76)
studied a lattice isomorphism on Sobolev spaces.
Later on many more mathematicians contributed to this research,
see details in \cite{V2005, V2012} and recent results 
on the subject in \cite{VN2015, V2016}. 
We also mention here recent paper \cite{HKM} on composition operator on Sobolev-Lorentz space.
 
Our work belongs to the second of the described directions. 
As of right now composition operators on $L_p$ have been  
investigated thoroughly enough (see \cite{CJ, SM, VU}). 
In the case of Lorentz spaces most of the research has been concerned with 
composition operators from $L_{p,q}$ to $L_{p,q}$, 
domain and image spaces having the same parameters
(see \cite{KK, ADV, KADA}). 
Here we initiate the study of a composition operator from $L_{r,s}$ to $L_{p,q}$,
where the parameters may differ.
The principal result of the paper is as follows. 
\begin{thm}\label{theorem:principal}
A measurable mapping $\varphi:X\to Y$ satisfying $\mathcal N^{-1}$-property 
induces a bounded composition operator 
$$
C_\varphi:L_{r,s}(Y)\to L_{p,q}(X), \qquad s\leq q
$$
if and only if 
$$
\int\limits_BJ_{\varphi^{-1}}(y)\, d\nu(y) \leq K^p\big(\nu(B)\big)^{\frac{p}{r}}
$$
for some constant $K<\infty$ and any measurable set $B$.
\end{thm}
We prove the theorem above in section \ref{comp}, 
while the range of composition operator and the case when composition operator is an isomorphism are studied in sections \ref{image} and \ref{iso}.

\section{Lorentz spaces}

Let $(X,\mathcal A, \mu)$ be a $\sigma$-finite measurable space.
The Lorentz space $L_{p,q}(X,\mathcal A, \mu)$  is the set of all measurable functions $f:X\to\mathbb{C}$
for which
$$
\|f\|_{p,q} = \Bigg(\frac{q}{p}\int\limits_{0}^{\infty}\big(t^{\frac{1}{p}}f^*(t)\big)^q\,\frac{dt}{t}\Bigg)^{\frac{1}{q}} < \infty, \quad \text{ if } 1<p<\infty, 1\leq q<\infty,
$$
or
$$
\|f\|_{p,\infty} = \sup\limits_{t>0}t^{\frac{1}{p}}f^*(t)<\infty, \quad \text{ if } 1<p<\infty, q=\infty.
$$
The \textit{non-increasing rearrangement} $f^*(t)$ of a function $f(x)$ is defined as
$$
f^*(t) = \inf\{\lambda>0 : \mu_{f}(\lambda)\leq t\},
$$  
where
$$
\mu_{f}(\lambda) = \mu\{x\in X : |f(x)|>\lambda\}
$$
is the \textit{distribution function} of $f(x)$.

Note that $\|\cdot\|_{p,q}$ is a norm if $1<q\leq p$ and a quasi-norm if $p<q$.
We will refer to $\|\cdot\|_{p,q}$ as the Lorentz norm. 
For brevity we will use $L_{p,q}(X)$ instead of $L_{p,q}(X,\mathcal A, \mu)$.

In what follows we will need the next properties of Lorentz spaces.
\begin{lem}[{\cite[Proposition 2.1.]{KKM}}] 
The Lorentz norm can be computed via distribution:
\begin{equation}\label{eq:norm_distr}
\|f\|_{p,q} = \Bigg(\frac{q}{p}\int\limits_{0}^{\infty}\big(t^{\frac{1}{p}}f^*(t)\big)^q\,\frac{dt}{t}\Bigg)^{\frac{1}{q}} 
= \Bigg(q\int\limits_{0}^{\infty}\left(\lambda\mu_{f}^{\frac{1}{p}}(\lambda)\right)^q\,\frac{d\lambda}{\lambda}\Bigg)^{\frac{1}{q}}
\end{equation}
and
\begin{equation}\label{eq:norm_distr_infty}
\|f\|_{p,\infty} = \sup\limits_{\lambda>0}\lambda\big(\mu_{f}(\lambda)\big)^{\frac{1}{p}}.
\end{equation}
\end{lem}
\begin{lem}[{\cite[equation (2.10)]{SW}}]
Let $E\subset X$ be a measurable set. 
The Lorentz norm of its indicator is 
\begin{equation}\label{eq:lemma_indicator_norm}
\|\chi_E\|_{p,q} = (\mu(E))^{\frac{1}{p}}. 
\end{equation}
\end{lem}
\begin{proof}
Observe that $\mu_{\chi_E}(\lambda) = \mu(E)\cdot\chi_{(0,1)}(\lambda)$. If $q<\infty$ we apply formula \eqref{eq:norm_distr}:
$$
\|\chi_E\|_{p,q} = \Bigg(q\int\limits_{0}^{1}\big(\lambda\mu(E)^{\frac{1}{p}}\big)^q\,\frac{d\lambda}{\lambda}\Bigg)^{\frac{1}{q}} = (\mu(E))^{\frac{1}{p}}.
$$ 

If $q=\infty$, we infer from \eqref{eq:norm_distr_infty}  that 
$$\|\chi_E\|_{p,\infty} = \sup\limits_{0<t<1}t\cdot (\mu(E))^{\frac{1}{p}} = (\mu(E))^{\frac{1}{p}}.$$
\end{proof}	

\begin{thm}[{\cite[Theorem 3.11]{SW}}]
Suppose that $f\in L_{p,q_1}$ and $q_1\leq q_2$, then $\|f\|_{p,q_2}\leq \|f\|_{p,q_1}$. 
\end{thm}

\section{Composition operator}\label{comp}
Let $(X,\mathcal A, \mu)$ and $(Y,\mathcal B, \nu)$ be  $\sigma$-finite measurable spaces and 
$\varphi:X\to Y$ be a measurable mapping.

\begin{lem}\label{lemma:principal}
Let $s\leq q$ and $f\circ\varphi \in L_{p,q}(X)$ for all $f\in L_{r,s}(Y)$,
then the following two statements are equivalent 

1. $\|f\circ\varphi\|_{p,q}\leq K\|f\|_{r,s}$ for any $f\in L_{r,s}(Y)$;

2. $(\mu(\varphi^{-1}(B)))^{\frac{1}{p}} \leq K(\nu(B))^{\frac{1}{r}}$ for any set $B\in\mathcal B$.
\end{lem}
\begin{proof}
Let $B\in\mathcal B$ and $\nu(B)<\infty$.
Plugging the indicator function $\chi_B(y)$ 
into statement 1  
and using property \eqref{eq:lemma_indicator_norm}, 
we obtain 2. 
If $\nu(B)=\infty$ the claim is trivial.

Suppose now that statement 2 holds. 
Let $f\in L_{r,s}(Y)$. 
First we find the expression for the distribution of the composition $f\circ\varphi$:
$$
\mu_{f\circ\varphi}(\lambda) = \mu(\{x\in X : |f(\varphi(x))|>\lambda\})
=\mu(\varphi^{-1}(\{y\in Y : |f(y)|>\lambda\})).
$$
Denote $E_\lambda = \{y\in Y : |f(y)|>\lambda\}$, then $\nu_f(\lambda) = \nu(E_\lambda)$
and $\mu_{f\circ\varphi}(\lambda) = \mu(\varphi^{-1}(E_\lambda))$.
From the inequality of statement 2 deduce
$$
\big(\mu(\varphi^{-1}(E_\lambda))\big)^{\frac{1}{p}} \leq K\big(\nu(E_\lambda)\big)^{\frac{1}{r}} 
$$
and thus
$$
\big(\mu_{f\circ\varphi}(\lambda)\big)^{\frac{1}{p}} \leq K\big(\nu_f(\lambda)\big)^{\frac{1}{r}}.
$$
Consequently,
\begin{multline*}
\|f\circ\varphi\|_{p,q}
= \Bigg(q\int\limits_{0}^{\infty}\left(\lambda\big(\mu_{f\circ\varphi}(\lambda)\big)^{\frac{1}{p}}\right)^q
\,\frac{d\lambda}{\lambda}\Bigg)^{\frac{1}{q}}\\
\leq
\Bigg(q\int\limits_{0}^{\infty}\left(\lambda K\big(\nu_{f}(\lambda)\big)^{\frac{1}{r}}\right)^q
\,\frac{d\lambda}{\lambda}\Bigg)^{\frac{1}{q}}
=K\|f\|_{r,q} \leq K\|f\|_{r,s}
\end{multline*}
if $s<\infty$, and
$$
\|f\circ\varphi\|_{p,\infty} 
= \sup\limits_{\lambda>0}\lambda\big(\mu_{f\circ\varphi}(\lambda)\big)^{\frac{1}{p}}
\leq K\sup\limits_{\lambda>0}\lambda\big(\nu_{f}(\lambda)\big)^{\frac{1}{r}}
= K\|f\|_{r,\infty}
$$
as desired. 
\end{proof}
Assuming $p=r, q=s$ obtain \cite[Theorem 1]{KK} and \cite[Theorem 2.1]{ADV} as consequences
of lemma \ref{lemma:principal}.

\begin{defn} 
A mapping $\varphi$ induces a \textit{composition operator} on Lorentz spaces
\begin{equation}\label{composition_operator}
C_\varphi:L_{r,s}(Y)\to L_{p,q}(X) \quad \text{ by the rule } C_\varphi f = f\circ\varphi
\end{equation}
whenever $f\circ\varphi\in L_{p,q}(X)$. 
\end{defn}
Clearly that $C_\varphi$ is a linear operator between two vector spaces.

A composition operator $C_\varphi$ is bounded if
\begin{equation}\label{eq:bounded_operator}
\|C_\varphi f\|_{p,q}\leq K\|f\|_{r,s}
\end{equation}
for every function $f\in L_{r,s}(Y)$, the constant $K$ being independent of the choice of~$f$.

Similarly, $C_\varphi$ is bounded below if
\begin{equation}\label{bounded-below}
\|C_\varphi f\|_{p,q} \geq k\|f\|_{r,s}.
\end{equation}

\begin{cor}\label{cor:LuzinN-1}
If a measurable mapping $\varphi$ induces a bounded composition operator, 
then $\varphi$ enjoys Luzin $\mathcal N^{-1}$-property 
(which means that $\mu(\varphi^{-1}(S))=0$ whenever $\nu(S)=0$).
\end{cor}
In particular, corollary \ref{cor:LuzinN-1} guarantees that if functions $f_1, f_2$ 
coincide a.e. on $Y$ then the images $C_\varphi f_1(x)$, $C_\varphi f_2(x)$
coincide a.e. on $X$.
On the other hand the a priori assumption of $\mathcal N^{-1}$-property enables us to 
consider \eqref{composition_operator} as an operator on equivalence classes.

Suppose we are given a measurable mapping $\varphi:X\to Y$
satisfying Luzin $\mathcal N^{-1}$-property. 
Then the measure $\mu\circ\varphi^{-1}$ is absolutely continuous with respect to $\nu$.
Thus the Radon--Nikodym theorem guarantees the existence of a measurable function 
$J_{\varphi^{-1}}(y)$ (the Radon--Nikodym derivative) such that
\begin{equation}\label{eq:rd1}
\mu(\varphi^{-1}(E)) = \int\limits_E J_{\varphi^{-1}}(y)\, d\nu(y).
\end{equation} 

On account of \eqref{eq:rd1}, theorem \ref{theorem:principal} follows 
immediately from lemma  \ref{lemma:principal}.

\begin{ex}
Let $X$ and $Y$ be subsets of $R^n$ with Lebesgue measure ${|\cdot|}$.
Consider a mapping  $\varphi:X\to Y$ such that the Jacobian is bounded $J(x,\varphi)<M<\infty$
and the Banach indicatrix\footnote{$N(y,f,X) = \#\{x\in X \mid f(x) = y\}$ is the number of elements of $f^{-1}(y)$ in $X$.}
is bounded $N(y,\varphi, X) < N$ as well.  
Therefore 
$$
\frac{N}{M} < J_{\varphi^{-1}}(y).
$$
Suppose that $\varphi$ induces a bounded operator from $L_{r,s}(Y)$ to $L_{p,q}(X)$
then by theorem~\ref{theorem:principal} 
and by the inequality above we obtain 
$$
0<\frac{N}{M}|B| < \int\limits_BJ_{\varphi^{-1}}(y)\, dy \leq K^p|B|^{\frac{p}{r}}
$$
and 
$$
\frac{N}{MK^p} <  |B|^{\frac{p}{r} - 1}.
$$
If we take a sequence of sets $B_k$ such that $|B_k|\to 0$ 
we will derive the necessary condition $p\leq r$, which is usually taken for granted.     
\end{ex}

\begin{ex}
Now let $X,Y\subset\mathbb R^2$. 
Examine a mapping $\varphi:X\to Y$ such that $\varphi(x_1,x_2) = (\frac{n}{2}, \frac{m}{2})$,
where $n-1<x_1<n$, $m-1<x_2<m$, $n,m\in \mathbb Z$. 
Let $\mu$ be the Lebesgue measure on $\mathbb R^2$ while $\nu$ be a discrete measure with atoms in 
$(\frac{n}{2}, \frac{m}{2})$, $n,m\in \mathbb Z$ and for the sake of simplicity
we set $\nu((\frac{n}{2}, \frac{m}{2})) = 1$. 
Then $J_{\varphi^{-1}}(y)=1$.
In this case the mapping $\varphi$ could induce a bounded composition operator
from $L_{r,s}(Y, \mathcal B, \nu)$ to $L_{p,q}(X, \mathcal A, \mu)$, even if $r<p$.
\end{ex}

\section{Properties of the image}\label{image}
In this section we exploit ideas from \cite{ADV} to investigate the  range of a composition operator.  
First we show that $J_{\varphi^{-1}}(y)$ may be assumed to be positive a.e. on $Y$.
Let $$Z=\{y\in Y : J_{\varphi^{-1}}(y)=0\},$$ then 
$$\mu(\varphi^{-1}(Z)) = \int\limits_Z J_{\varphi^{-1}}(y)\, d\nu(y) = 0.$$
Thus, after redefining the map  $\varphi$ on the set $\mu(\varphi^{-1}(Z))$ of measure zero
we obtain the property $J_{\varphi^{-1}}(y)>0$ a.e on $Y$.

\begin{thm}\label{theorem:bounded_ below}
A measurable mapping $\varphi$ satisfying $\mathcal N^{-1}$-property 
induces a bounded below composition operator 
$$
C_\varphi:L_{r,s}(Y)\to L_{p,q}(X), \quad s\geq q
$$
if and only if
\begin{equation}\label{eq:rd-below}
\int\limits_BJ_{\varphi^{-1}}(y)\, d\nu(y) \geq k^p\big(\nu(B)\big)^{\frac{p}{r}}
\end{equation}
for any $B\in \mathcal B$.
\end{thm}
\begin{proof}
Applying \eqref{bounded-below} to the indicator function $\chi_B(y)$ and using \eqref{eq:lemma_indicator_norm},
\eqref{eq:rd1} we derive
$$
\bigg(\int\limits_BJ_{\varphi^{-1}}(y)\, d\nu(y)\bigg)^{\frac{1}{p}} \geq k\big(\nu(B)\big)^{\frac{1}{r}}.
$$ 

Suppose now \eqref{eq:rd-below} holds.
Then in view of \eqref{eq:rd1}
$$
\mu_{f\circ\varphi}(\lambda) = \int\limits_Y\chi_{E_\lambda}(y) J_{\varphi^{-1}}(y)\, d\nu(y)
\geq k^p\big(\nu(E_\lambda)\big)^{\frac{p}{r}} = k^p\big(\nu_f(\lambda)\big)^{\frac{p}{r}}.
$$
Thus $\big(\mu_{f\circ\varphi}(\lambda)\big)^{\frac{1}{p}}\geq k\big(\nu_f(\lambda)\big)^{\frac{1}{r}}$ and
$$
\|C_\varphi f\|_{p,q} \geq k\|f\|_{r,q} \geq k\|f\|_{r,s}.
$$
\end{proof}

Let $s=q$. Making use of the well known fact from functional analysis,
which says that a linear bounded operator between Banach spaces 
is bounded below if and only if it is one-to-one and has closed range, 
we arrive to the following assertion.
\begin{thm}\label{theorem:closed_image}
A bounded composition operator $C_\varphi:L_{r,s}(Y)\to L_{p,s}(X)$ is injective and has the closed image
if and only if there is a constant $k>0$ such that
$$
\int\limits_BJ_{\varphi^{-1}}(y)\, d\nu(y) \geq k^p\big(\nu(B)\big)^{\frac{p}{r}}
$$
for any $B\in \mathcal B$. 
\end{thm}

Next we discuss where a bounded composition operator has dense image.
\begin{thm}\label{theorem:dense_image}
The image of a bounded composition operator 
$C_\varphi:L_{r,s}(Y, \mathcal B, \nu)\to L_{p,q}(X, \mathcal A, \mu)$ 
is dense in  $L_{p,q}(X, \varphi^{-1}(\mathcal B), \mu)$.
\end{thm}
\begin{proof}
Let $\chi_A\in L_{p,q}(X, \varphi^{-1}(\mathcal B), \mu)$ be the indicator function of a set $A=\varphi^{-1}(B)$,
$B\in\mathcal B$.
It is easy to see that  $\chi_A(x) = \chi_B(\varphi(x))$, though we cannot ensure  
$\chi_B(y)\in L_{r,s}(Y, \mathcal B, \nu)$.

Let $B = \bigcup B_k$, where $\{B_k\}$ is an increasing sequence of sets of finite measure.   
Then $\chi_{B_k}(y)\in L_{r,s}(Y, \mathcal B, \nu)$. Denote $f_k=C_\varphi\chi_{B_k}$.
Obviously $f_k(x)\leq \chi_A(x)$ and $f_k(x)\to \chi_A(x)$ as $k\to \infty$ a.e. on $X$.
The similar inequality and convergence take place for distributions ($\mu_{f_k}$ and $\mu_{\chi_A}$),
therefore from the Lebesgue theorem $f_k(x)\to \chi_A(x)$ in $L_{p,q}(X)$. 
The same arguments work for simple functions. 

It follows that every simple function from  $L_{p,q}(X, \varphi^{-1}(\mathcal B), \mu)$
is the limit of images.
Since the set of simple functions is dense in $L_{p,q}(X)$ we conclude that the image 
$C_\varphi(L_{r,s}(Y, \mathcal B, \nu))$ is dense in $L_{p,q}(X, \varphi^{-1}(\mathcal B), \mu)$. 
\end{proof}

\section{Isomorphism}\label{iso}
We will say  that a mapping $\varphi:X\to Y$ induces an isomorphism 
of Lorentz spaces $L_{p,q}(Y, \mathcal B, \nu)$, $L_{p,q}(X, \mathcal A, \mu)$
whenever the composition operator $C_\varphi$ is bijective and the inequalities 
\begin{equation}\label{bounded_below_above}
k\|f\|_{p,q} \leq \|C_\varphi f\|_{p,q}\leq K\|f\|_{p,q}
\end{equation}
hold for every function $f\in L_{p,q}(Y, \mathcal B, \nu)$
and for some constants $0<k\leq K<\infty$ independent of the choice of $f$. 

\begin{thm}
A measurable mapping satisfying $\mathcal N^{-1}$-property 
induces an isomorphism of Lorentz spaces
$$
C_\varphi:L_{p,q}(Y, \mathcal B, \nu)\to L_{p,q}(X, \mathcal A, \mu)
$$
if and only if 
\begin{equation}\label{rd_below_above}
k^p \leq J_{\varphi^{-1}}(y) \leq K^p \quad \text{ a.e. } y\in Y
\end{equation}
and 
$\varphi^{-1}(\mathcal B) = \mathcal A$.
\end{thm}
\begin{proof}
Let $\varphi$ induce an isomorphism.
Thanks to theorems \ref{theorem:principal}, \ref{theorem:bounded_ below} 
inequalities \eqref{rd_below_above} are a straightforward consequence of \eqref{bounded_below_above}.

Show that $\varphi^{-1}(\mathcal B) = \mathcal A$. Let $A\in\mathcal A$ and $\mu(A)<\infty$, 
then the indicator function $\chi_A\in L_{p,q}(X, \mathcal A, \mu)$.
Because of the surjectivity there is a function $f\in L_{p,q}(Y, \mathcal B, \nu)$ such that
$\chi_A = C_\varphi f$. 
Observe that the set $B=\{y\in Y : f(y) = 1\}$ is an element of $\mathcal B$ and $f = \chi_B$. 
This yields $\chi_A = C_\varphi \chi_B = \chi_{\varphi^{-1}(B)}$ and hence $A=\varphi^{-1}(B)$.
Thus $\mathcal A = \varphi^{-1}(\mathcal B)$.

Now assume that $\mathcal A = \varphi^{-1}(\mathcal B)$ and \eqref{rd_below_above} holds.
Again \eqref{bounded_below_above} is equivalent to \eqref{rd_below_above} 
owing to theorems \ref{theorem:principal}, \ref{theorem:bounded_ below}. 
From theorem \ref{theorem:closed_image} we infer that the operator $C_\varphi$ is one-to-one and 
the image $C_\varphi(L_{p,q}(Y, \mathcal B, \nu))$ is closed, 
whereas theorem \ref{theorem:dense_image} implies the density of the image in 
$L_{p,q}(X, \mathcal A, \mu)$. 
Consequently $C_\varphi(L_{p,q}(Y, \mathcal B, \nu))=L_{p,q}(X, \mathcal A, \mu)$.   
  
\end{proof}

\end{document}